\RequirePackage{fix-cm}
\documentclass[smallextended]{svjour3}       
\smartqed  
\usepackage{graphicx}
\makeatletter
\let\cl@chapter\undefined
\makeatletter

\usepackage{enumerate, eucal, amssymb,hhline}
\usepackage[all, arrow,arc]{xy}
\usepackage{graphicx}
\usepackage{amssymb}
\usepackage{amsmath}
\usepackage{amsfonts}
\usepackage{booktabs}
\usepackage{multirow}
\usepackage{latexsym}
\usepackage{enumerate}
\usepackage{hyperref}
\usepackage{cleveref}
\usepackage{microtype}
\usepackage{algorithm}
\usepackage[noend]{algorithmic}
\usepackage{color}
\allowdisplaybreaks[4]
\usepackage{natbib}
\usepackage{mathptmx}
\newtheorem{defn}{Definition}[section]

\newtheorem{thm}[defn]{Theorem}
\newtheorem{prop}[defn]{Proposition}
\newtheorem{cor}[defn]{Corollary}
\newtheorem{lem}[defn]{Lemma}

 \def\FF{{\mathbb F}}

\DeclareMathOperator{\Mult}{Mult}

\DeclareMathOperator{\Aut}{Aut}

\def\F{\mathbb F}

\begin{document}

\title{Algebraic Properties of Subquasigroups and Construction of Cryptographically Suitable Finite Quasigroups
}

\titlerunning{Construction of Cryptographically Suitable Finite Quasigroups}        

\author{V.A. Artamonov        \and
        Sucheta Chakrabarti \and
        Sharwan K. Tiwari \and
        V.T. Markov
}


\institute{V.A. Artamonov \at
              Department of Algebra, Faculty of Mechanics and Mathematics, Moscow
 State University 
              \email{artamon@mech.math.msu.su}           
           \and
           Sucheta Chakrabarti \at
              Scientific Analysis Group, DRDO, Delhi\\
              \email{suchetadrdo@hotmail.com, suchetachakrabarti@sag.drdo.in}           
           \and
           Sharwan K. Tiwari \at
              Scientific Analysis Group, DRDO, Delhi\\
              \email{shrawant@gmail.com, sktiwari@sag.drdo.in}           
           \and
           V.T. Markov \at
              Department of Algebra, Faculty of Mechanics and Mathematics, Moscow
 State University 
}


\maketitle

\begin{abstract}
In this paper, we identify many important properties and develop criteria for the existence of subquasigroups in  finite quasigroups. Based on these results, we propose an effective method that concludes the nonexistence of subquasigroup of a finite quasigroup, otherwise finds its all possible proper subquasigroups.  
This has an important application in checking the cryptographic suitability of  a quasigroup. 

Further, we propose a binary operation using arithmetic of finite fields to construct quasigroups of order $p^r$. We develop the criteria under which these quasigroups have desirable cryptographic properties, viz. polynomially completeness and possessing no proper subquasigroups. Then a practical method is given to construct cryptographically suitable quasigroups. We also illustrate these methods by some academic examples and implement all proposed algorithms in the computer algebra system {\sc{Singular}}.  
\end{abstract}

 \section{Introduction}
Identifying some algebraic properties in quasigroups, such as polynomially completeness, having no proper subquasigroups, is vital for analyzing and designing quasigroup-based crypto primitives and schemes. Testing these properties has a two-fold advantage; on the one hand, we can find cryptographically suitable quasigroup for the design. On the other hand, the robustness of already  existing quasigroup based crypto schemes can be checked.  

Quasigroup based cryptography is an emerging research area for efficient secure communication. It has an advantage in resource constraint environments such as smart cards, RFID systems etc. 
The nonassociative and noncommutative properties of the quasigroup structure make them useful in the crypto designs,  because many well-known cryptanalysis techniques to attack these crypto designs  would not be directly applicable. Quasigroup is a well-studied algebraic combinatoric structure, see in \cite{Belyavskaya1989, Belyavskaya1992, Keedwell2015, Denes1991, kepka1978, kepka1971}. However, its intensive use by the crypto community can be traced back only in the last two decades, see \cite{Belyavskaya1994, viktor2017, smith2007}. The significant properties of quasigroups for the crypto designs are polynomial completeness (simplicity, non-affineness), degree of non-commutativity and non-associativity, non-existence of proper subquasigroup, see \cite{Artamonov2012, Artamonov2013, Artamonov2017, Galatentko2018, Otokar2012}.

Quasigroup string transformations are one of the fundamental building blocks in the design of quasigroup based crypto algorithms and the security of these transformations mainly depends on the proper choice of quasigroups \cite{Artamonov2016, Artamonov2020, smile1997, Markovski1999}.  In particular, the choice of quasigroup affects the rate of growth of the period and randomness of the output of quasigroup string transformations \cite{Dimitrova04}.  A significant period and randomness provide good security to the crypto algorithms. Furthermore, these properties ensure that the search space for the brute-force attack can not be reduced; and the problem of solving an equivalent algebraic system of such quasigroups based crypto algorithm is NP-Complete \cite{Horv2008}. 
  
Developing efficient methods to check and construct cryptographically suitable quasigroups of  finite order is ongoing research. In this paper, we work on two important problems related to the quasigroup-based cryptography. At first, we develop methods to conclude the non-existence of proper subquasigroup and then propose a class of cryptographically suitable quasigroups of finite order $p^r$. In Section \ref{s1}, we revise the main definitions and properties of quasigroups used in this paper. Then, in Section \ref{s2}, we prove some important properties of subquasigroups. Based on these results, we propose an effective algorithm to determine all the possible proper subquasigroups of any given finite quasigroup, otherwise deduce the non-existence of any proper subquasigroup (the desirable property for crypto schemes). In section \ref{s3}, we propose  a method for the construction of cryptographically suitable finite quasigroup of order $p^r$. At first, we introduce a quasigroup binary operation based on the arithmetic of finite fields and then develop the criteria under which these quasigroups are polynomially complete and  possess no subquasigroups.  We then finalize all these findings in an effective algorithm to construct quasigroups of given finite orders for any prime $p$ and integer $r$. In particular, the algorithm constructs the cryptographically desirable quasigroups of order $2^r$. Finally, the given methods are elaborated by their implementation in Singular \cite{singular} on some examples.

 \section{ Preliminaries}\label{s1}
 In this section, we introduce some of the terminology used in this paper.\\
 A \emph{quasigroup} is a set $Q$ with a binary operation of multiplication
  $xy$ such that for all $a,b\in Q$ the equations 
 $
 ax=b,\quad ya=b
 $
 have unique solutions 
 $
 x=a\diagdown b ,\quad y=b \diagup a.
 $
 Then the class of quasigroups form a variety of algebras with three operations 
 $xy,\; x\diagdown y ,\quad x \diagup y$ 
  which is defined by identities
 \begin{eqnarray} \label{eq1}
 (xy )\diagup y=x=  (x \diagup y)y , 
 \quad 
x \diagdown (xy)=y=  
  x(x\diagdown y).
 \end{eqnarray}

Each quasigroup $Q$ can be given by a Latin square 
\begin{eqnarray} \label{eqls1}
  \begin{array}{ c|ccc}
  \cdot & x_1 & \hdots & x_n\\
  \hline
 x_1 &   a_{11}  & \hdots & a_{1n} \\ 
\vdots &  \hdots & \hdots  &  \hdots\\
 x_n &   a_{n1}   & \hdots & a_{nn} \\
  \end{array}
  \end{eqnarray}
  of size $n$.
  The elements of $Q$ are   $\{x_1,\ldots,x_n\}$,
  each entry $a_{ij}$ stands for the product $x_iy_j$ in the quasigroup $Q$.
  
Let $x\cdot y, x*y$ be two quasigroup multiplications    on a set $Q$.
 We say that  multiplication $x*y$ is an
   \textit{isotope} of multiplication $x\cdot y$ if
there exist permutations $\pi, \pi_1,\pi_2$ on $Q$ such that 
\begin{eqnarray} \label{isot}
x*y = \pi \left(\pi_1^{-1}(x)\cdot \pi_2^{-1}(y)\right)
\end{eqnarray}
for all $x,y\in Q$.

 In terms of the Latin square  \eqref{eqls1} it means that  we replace it by the square 
 \begin{eqnarray*} \label{eqls0}
  \begin{array}{c|ccc} 
  \ast&  x_1   & \hdots &   x_n \\
  \hline
 x_1 &   b_{11}  & \hdots & b_{1n} \\ 
\vdots &  \hdots & \hdots  &  \hdots\\
 x_n  &   b_{n1}   & \hdots & b_{11} \\
  \end{array}
  \end{eqnarray*}
 where 
 $$
 b_{ij} = \pi  \left(\pi_1^{-1}(x_i)\cdot \pi_2^{-1}(x_j)\right)=
 \pi \left(a_{\pi_1^{-1}(x_i),\pi_2^{-1}(x_j)}\right).
 $$
 In other words, we   rearrange columns and  rows of $(Q,\cdot)$ using a permutations
  $\pi_2$ and $\pi_1$, respectively, and  afterwards  permute elements
   of the obtained   Latin square using $\pi^{-1}$.

The left and right multiplications by an element, namely 
$$L_{x}(y)=x\ast y \text{ and } R_y(x)=x\ast y$$
induce permutations on $Q$. Note that a left multiplication $L_{x_i}$  is the permutation of $Q$ represented by the $i^{th}$ row of Latin square given in (\ref{eqls1}) and denoted by $\sigma_i$. Similarly, $\tau_j$  represents the column permutation of (\ref{eqls1}) induced by right multiplication $R_{x_j}$. 
\\

The multiplication group Mult(Q) on a quasigroup $(Q,\ast)$ of order $n$ is a permutation group generated by the all row and column permutations of the Latin square of $Q$. That is,
\[\Mult(Q)=\langle \sigma_i, \tau_i \mid 1\leq i\leq n\rangle\]  
We denote by $G(Q)$, the following subgroup of $\Mult(Q)$,
\[G(Q)=\langle \sigma_i\sigma_{j}^{-1}, \tau_i\tau_j^{-1} \mid 1\leq i,j\leq n\rangle\]  

\begin{defn}
A finite quasigroup $(Q,\ast)$ is  affine if $Q$   is  equipped with a structure of additive Abelian group $(Q,+)$ such that, for each $x,y\in Q$, 
$$x\ast y=\alpha x+\beta y +c, ~~ \text{where } c\in Q$$
and $ \alpha,\beta $ are automorphisms of $(Q,+)$.

\end{defn}
\begin{defn}
A finite quasigroup  is simple if  it has only trivial congruences.
\end{defn}
\begin{thm}\cite{Artamonov2013, Hagemann1982}
A finite quasigroup  is polynomially complete if and only if it is simple and non-affine. 
\end{thm}

\section{ Subquasigroups}\label{s2}

\begin{prop} \label{sq}
Let $Q$ be a finite quasigroup and $H$ be a subset of $Q$ closed under
multiplication. Then $H$ is a subquasigroup in $Q$.
\end{prop}
\begin{proof}
 Let $a,b\in H$. We need to prove that the equations $ax=b,\, ya=b$ have solutions in $H$.
 Note that the   left multiplication  $L_a$ maps $H$ into $H$.
  The map $L_a$ has a finite order $d$ in $Q$. Hence $L^{-1}_a=L_a^{d-1}$,
 Therefore $H$ contains the element
 $$
 L_a^{d-1}b = L^{d-1}_a(ax)=L^d_ax=x.
 $$ 
Similarly $y\in H$.
\end{proof}

Let a quasigroup $Q$ with elements $x_1,\ldots, x_n$  contains a subquasigroup 
$W$ with elements $x_1,\ldots,x_k$. Then the  Latin square \eqref{eqls1} of $Q$ has the form 
\begin{eqnarray} \label{eqls1s}
  \begin{array}{c|cccccc} 
 \ast & x_1 & \hdots   & x_k & x_{k+1} & \hdots  & x_n\\
  \hline
 x_1 &   a_{11}  & \hdots      & a_{1k} & a_{1,k+1} & \hdots  & a_{1n} \\ 
\vdots &  \hdots & \hdots    & \hdots & \hdots  & \hdots       &  \hdots\\
 x_k &   a_{k1}  & \hdots      & a_{kk} & a_{k,k+1} & \hdots  & a_{kn} \\ 
 x_{k+1} &   a_{k+1,1}  & \hdots      & a_{k+1,k} & a_{k+1,k+1} & \hdots  & a_{k+1,n} \\ 
\vdots &  \hdots & \hdots    & \hdots & \hdots  & \hdots       &  \hdots\\
 x_n &   a_{n1}   & \hdots      &  a_{nk} & a_{n,k+1} & \hdots & a_{nn} \\
  \end{array}\; 
  \end{eqnarray}
where 
$$
a_{ij} \in \begin{cases} 
 W ,   &   \text{ if } 1\leqslant i,j\leqslant k;\\
 Q\setminus W,  &  \text{ if  either } i>k, \, j\leqslant k, \text{ or } j>k ,\, i\leqslant k.
 \end{cases}
 $$
 It follows immediately that  $W$  is contained in each
 column and each row
 $$
 \begin{pmatrix}
 a_{k+1,j}\\
 \vdots \\
 a_{n,j}
 \end{pmatrix}, \quad 
 \begin{pmatrix}
 a_{j,k+1} & \hdots & a_{jn}
 \end{pmatrix}
  \text{ for any } j=k+1,\ldots,n.
  $$
  In particular $n-k\geqslant |W|$.
  \begin{prop}[\cite{Wall1957}] \label{sq11}
  Let $W$ be a subquasigroup in $Q$. Then $|W|\leqslant \frac{|Q|}{2}.$
  \end{prop}
 \begin{prop} \label{sq4}
 Let $Q$ be a quasigroup of order 4. Then the following are equivalent:
 \begin{itemize}\vspace{-2mm}
 \item[1.] \label{cq1}
 $Q$ has a subquasigroup  $W$ such that $1< |W|<4$;
 \item[2.] \label{cq2}
 $W$ is a congruence class of order 2 and  therefore $Q$ is not simple.
 \end{itemize}
\end{prop}
\begin{proof}
Let $1< \lvert W\rvert<4$. Then $|W|=2$ by  Proposition \ref{sq11}.
Hence, $W=\{x_i,x_j\}$ is a cyclic group of order 2. 
\\
Let $W=\{x_i,x_j\}\subset Q=\{x_i,x_j,x_u,x_v\}$. The Latin square of $Q$ is 
\begin{eqnarray*} \label{eqls1}
  \begin{array}{ c|cc|cc}
  \cdot & x_i &x_j&x_u&x_v\\
  \hline
 x_i &   a_{ii}  & a_{ij}&a_{iu}&a_{iv} \\ 
 x_j &   a_{ji}  & a_{jj}&a_{ju}&a_{jv} \\  \hline
 x_u &   a_{ui}  & a_{uj}&a_{uu}&a_{uv} \\
 x_v &   a_{vi}  & a_{vj}&a_{vu}&a_{vv} \\ 
  \end{array}
  \end{eqnarray*}
Clearly, upto rearrangement of $x_i$ and $x_j$, we have 
$$
x_i^2=x_i,~ x_ix_j=x_jx_i=x_j, ~x_j^2=x_i,~ i.e. ~ a_{it}, a_{jt}\in W ~ \text{for } t=i,j
$$ 
Therefore, 
$$\{x_i,x_j\}\cdot \{x_u,x_v\} ~\& ~\{x_u,x_v\}\cdot \{x_i,x_j\} \in \{x_u,x_v\} ~\text{ and }~\{x_u,x_v\}\cdot \{x_u,x_v\}\in \{x_i,x_j\}$$
Let us consider a partition of $Q$, $\mathcal{P}=\big\{\{x_i,x_j\}, \{x_u,x_v\}\big\}$. Now, we have 
\begin{eqnarray*}
\{x_i,x_j\}\cdot \{x_i,x_j\}&=&\{x_i,x_j\}=\{x_u,x_v\}\cdot \{x_u,x_v\},\\
\{x_u,x_v\}\cdot \{x_i,x_j\}&=&\{x_u,x_v\}=\{x_i,x_j\}\cdot \{x_u,x_v\},
\end{eqnarray*}
That is, $Q/\mathcal{P}$ is a group of order $2$ with the identity element $\{x_i,x_j\}$. Hence, the partition $\mathcal{P}$ is a congruence relation. 
  
Converse statement follows from \cite[Theorem 7]{Artamonov2013}.
\end{proof}  

\begin{thm}[Recognition of subquasigroups] \label{rsq}
	Let  $(Q=\{x_1,\ldots,x_n\},\ast)$ be a quasigroup of order $n$. Take 
	a subset $W=\{x_{i_1},\ldots,x_{i_k}\},\, k\leqslant \frac{n}2,$ in $Q$  which 
	is a union of orbits of any permutation $\sigma_{x_{i_j}}$ for   $j=1,\ldots, k$. 
	 $W$ is a subquasigroup in $Q$ if and only if 	  	$x_{i_j}\in W$ implies  $\sigma_{x_{i_t}}( x_{i_j})= x_{i_t}\ast x_{i_j} \in W$ for every $t=1,\ldots,k$.
\end{thm}
\begin{proof}
Let $x_{i_{j}}, x_{i_{j'}}\in W$.   Then     $x_{i_{j}}\ast x_{i_{j'}}=\sigma_{x_{i_{j}}}(x_{i_{j'}})\in W$
 because $\sigma_{x_{i_{j}}}(x_{i_{j'}})$ belongs to an orbit of $x_{i_{j'}}$ which is contained in  $W$.   Thus $W$ is a subquasigroup.
 
 Conversely if   $W$ is a subquasigroup in the quasigroups $Q$
  with the Latin square  \eqref{eqls1s}
  take $C_i$ as the set of enties of  $ i^{\text{th}}$ row 
  $\begin{pmatrix} a_{i1}  & \hdots & a_{ik}\end{pmatrix}$
  which is equal to $W$.
 Then all properties  
 are satisfied.
 \end{proof}
 Note that   the orbit of each $x_{i_j}$ with
  respect to $\sigma_{x_{i_t}}$ for all $t=1,\ldots,k$  is contained in $W$. Moreover $W$ is a union of all orbits
  $$\bigcup_{1\leqslant t,j \leqslant k,\, ~s\geqslant 0} (\sigma^s_{x_{i_t}}(x_{i_j})).$$
 
Now, we illustrate an efficient procedure to obtain all proper subqusigroups of a quasigroup $(Q,\ast)$ based on Theorem \ref{rsq} by means of some concrete examples.

\begin{example}\label{Example1}

Let us consider a quasigroup $(Q=\{1,2,\ldots,8\},\ast)$ of order $8$
with the following Latin square, $L$
\begin{align*}
\begin{array}{ c|cccccccc} 
  \ast &1&2&3&4&5&6&7&8\\ 
      \hline
       1&2&1&4&3&5&6&7&8\\  
       2&1&3&2&4&6&7&8&5\\  
       3&4&2&3&1&7&8&5&6\\  
       4&3&4&1&2&8&5&6&7\\  
       5&5&6&7&8&1&2&3&4\\  
       6&6&7&8&5&2&3&4&1\\  
       7&7&8&5&6&3&4&1&2\\
       8&8&5&6&7&4&1&2&3 
 \end{array}
\end{align*}
Let's consider the first  row permutation, say $\sigma_1$, of the Latin square $L$. We use two distinct sets, denoted by $\Sigma$ and $\Sigma'$, to store the indices of the row permutations. The set $\Sigma$ is composed of  all indices $j$, if $\sigma_j$ are being considered in the method. Whereas, the set $\Sigma'$ is composed of some particular indices (not by all) which will be illustrated later. Initially, $\Sigma'=\emptyset$ and since we are considering the first permutation, $\Sigma=\{1\}$. 
Now, we start by decomposing  $\sigma_1$ into disjoint cycles. If $O_{ij}$ denotes the $j^{th}$ cycle of $\sigma_i$ row permutation, then the cycles of $\sigma_1$ are
\[O_{11}=(1,2),~\, O_{12}=(3,4),~\, O_{13}=(5),~\, O_{14}=(6),~\, O_{15}=(7),~\, O_{16}=(8)\]
By Theorem \ref{rsq}, if a subset $W=\{i_1,i_2,\ldots,i_k\}\subset Q$ is a subquasigroup, then $W$ is a union of cycles of any permutation $\sigma_{i_j}$ such that $i_{j}\in O_{ik}\subset W$ and $\lvert W\rvert \leq \frac{\lvert Q\rvert}{2}$. So, first we locate the cycle, say $O_{ik}$, which has $i^{th}$ element of $Q$ and consider all possible subsets formed by the unions of cycles with $O_{ik}$ such that the cardinality of each subset is not grater than $\frac{\lvert Q\rvert}{2}$ and their elements do not belong to $\Sigma'$. Mathematically, we can express the set $\mathcal{P}$  of all possible candidates in the following way
\[\mathcal{P}=\Big\{W\in \mathcal{P'}\bigm| \lvert W\rvert +\lvert O_{ik}\rvert \leq n/2 \Big\}, \text{ where}\]
\[\mathcal{P'}=\text{PowerSet}\big\{O_{ij\mid j\neq k} \bigm| O_{ij}\subset Q\setminus \Sigma' ~\&~\lvert O_{ij}\rvert +\lvert O_{ik}\rvert \leq n/2   \big\}\] 
So in the above example, the set denoted by $\mathcal{P}$, of all possible candidates for being subquasigroup of $Q$ corresponding to $\sigma_1$, is given by
\begin{align*}
\begin{array}{c}
\hspace{-2mm}\mathcal{P}=\Big\{\{O_{11}\},\{O_{11}\cup O_{12}\}, \{O_{11}\cup O_{13}\},\{O_{11}\cup O_{14}\}, \{O_{11}\cup O_{15}\},\\ \{O_{11}\cup  O_{16}\}, 
\{O_{11}\cup O_{13}\cup O_{14}\},\{O_{11}\cup O_{13}\cup O_{15}\}, \{O_{11}\cup  O_{13} \\ \cup O_{16}\},
 \{O_{11}\cup O_{14}\cup O_{15}\},\{O_{11}\cup O_{14}\cup O_{16}\},\{O_{11}\cup O_{15}\cup O_{16}\} \Big\}
\end{array}
\end{align*}
i.e.
\begin{align*}
\begin{array}{c}
\mathcal{P}=\Big\{\{1,2\},\{1,2,3,4\}, \{1,2,5\},\{1,2,6\}, \{1,2,7\}, \{1,2,8\}, \{1,2,5,6\},\\
 \{1,2,5,7\}, \{1,2,5,8\}, \{1,2,6,7\}, 
 \{1,2,6,8\}, \{1,2,7,8\} \Big\}
\end{array}
\end{align*}

If a subset $W \in \mathcal{P}$ is closed under the operation $\ast$ then $W$ is a subquasigroup of $Q$. For instance in the ongoing example, only the subset $\{1,2,3,4\}$ of $\mathcal{P}$ is closed. So the list $\mathcal{L}$ of subquasigroup obtained by considering the permutation $\sigma_1$ is
\[\mathcal{L}=\Big\{ \{1,2,3,4\}\Big\}\]

Now, we find the subquasigroups consisting of only elements from the set $S=\bigcup_{W\in \mathcal{L}} W\setminus \Sigma=\{2,3,4\}$. That is, we need to consider only those orbits  of permutations $\sigma_j, j\in S$ for the possible subquasigroups which are subsets of $S$. We repeat the process until $S$ is empty.   

Here, on considering $2\in S$,
\[O_{21}=(1),~\, O_{22}=(2,3),~\, O_{23}=(4),~\, O_{24}=(5,6,7,8)\]
Since $O_{22},O_{23}\subset S$,
\begin{align*}
\begin{array}{c}
\mathcal{P}=\Big\{\{2,3\},\{2,3,4\}\Big\}
\end{array}
\end{align*}
The subset $\{2,3\}$ is closed w.r.t. $\ast$. So, the subquasigroup list is
\[\mathcal{L}=\Big\{ \{1,2,3,4\}, \{2,3\}\Big\}\]
The updated  sets $S=S\setminus\{2\}=\{3,4\}$ and $\Sigma=\Sigma\cup\{2\}=\{1,2\}$. Next, conside $3\in S$,
\[O_{31}=(1,4),~\, O_{32}=(2),~\, O_{33}=(3),~\, O_{34}=(5,7),~\, O_{35}=(6,8)\]
So, $\mathcal{P}=\big\{ \{3\}\big\}$ and $\{3\}$ is closed.
Hence the subquasigroups are 
\[\mathcal{L}=\Big\{ \{1,2,3,4\}, \{2,3\},\{3\}\Big\}\]
Now, $S=S\setminus\{3\}=\{4\}$ and $\Sigma=\Sigma\cup\{3\}=\{1,2,3\}$. Decomposition of $\sigma_4$ provides
\[O_{41}=(1,3),~\, O_{42}=(2,4),~\, O_{43}=(5,6,7,8)\]
Since $O_{42}\nsubseteq S$, $\mathcal{P}=\emptyset$. The sets $S=S\setminus\{4\}=\emptyset$ and $\Sigma=\Sigma\cup\{4\}=\{1,2,3,4\}$.

At this step, we recall that  we started with first permutation $\sigma_1$ and then based on obtained subquasigroups we consider other permutations for finding the subquasigroups consisting only elements form the set $S$. The set $\Sigma'$ contains indices of such considered permutations $\sigma_1$, $\Sigma'=\{1\}$.

Now, one has to consider the next permutation, in the process, is $\sigma_i$, where   
$$i = \min(Q\setminus \Sigma).$$

So for the example, we are illustrating here, the index for the next permutation
$$i=\min\Big\{ \{1,2,\ldots,8\}\setminus \{ 1,2,3,4\}  \Big\}= \min\Big\{ 5,6,7,8\Big\}=5.$$
Now we repeat the above described complete process starting with the permutation $\sigma_5$. First get the cycle decomposition of $\sigma_5$, which is
\[O_{51}=(1,5),~\, O_{52}=(2,6),~\, O_{53}=(3,7),~\,O_{54}=(4,8)\]
Since, $5 \in O_{51}$ and $O_{51}\nsubseteq Q\setminus \Sigma'$, the set of possible candidates $\mathcal{P}=\emptyset$.
The sets $\Sigma=\Sigma\cup\{5\}=\{1,2,3,4,5\}$ and $S=\bigcup_{W\in \mathcal{L}} W\setminus \Sigma=\emptyset $.   $\Sigma'=\{1,5\}$. 

By continuing, the index of next row permutation to consider,
$$i=\min(Q\setminus \Sigma)= 6$$
and the cycle decomposition of $\sigma_6$ is
\[O_{61}=(1,3,6,8),~\, O_{62}=(2,4,5,7)\]
Again $O_{61}\nsubseteq Q\setminus\Sigma'$, so $\mathcal{P}=\emptyset$. The sets  $\Sigma=\{1,2,3,4,5,6\}$, $S$ remains $\emptyset$, and $\Sigma'=\Sigma' \cup\{6\}=\{1,5,6\}$.

The next row index, $i=\min(Q\setminus \Sigma)= 7$. The cycle decomposition of $\sigma_7$ is
\[O_{71}=(1,7),~\, O_{72}=(2,8),~\,O_{73}=(3,5),~\, O_{74}=(4,6)\]
So, once again $O_{71}\nsubseteq Q\setminus \Sigma'$, $\mathcal{P}=\emptyset$. Then, $\Sigma=\{1,2,3,4,5,6,7\}$ and $\Sigma'=\Sigma' \cup\{7\}=\{1,5,6,7\}$

Now, $t=\min(Q\setminus \Sigma)= 8.$ The decomposition of $\sigma_8$
\[O_{81}=(1,3,6,8),~\, O_{82}=(2,4,5,7)\]
$O_{81}\nsubseteq Q\setminus\Sigma'$.  
All the row permutations have been considered, hence the procedure terminates and the list of all proper subquasigroups of $(Q,\ast)$,
\[\mathcal{L}=\Big\{ \{1,2,3,4\}, \{2,3\},\{3\}\Big\}\]
\end{example}
Below we display the pseudo code of the above described method to find out all proper subquasigroups of a given quasigroups.
\begin{algorithm}[H]
\caption{Finding all Proper Subquasigroups}
\label{subqgalgo}
\begin{algorithmic}[1]
\REQUIRE A Latin square $L$ of Quasigroup $(Q=\{1,2,\ldots,n\}, \ast)$) 
\ENSURE List of all  subquasigroups
\STATE list $\mathcal{L}=\emptyset$, $i=1$, $\Sigma =\emptyset$, $\Sigma' =\emptyset$ 
\WHILE{$i\leq n$}
\STATE $\Sigma=\Sigma \cup \{i\}$
\STATE decompose $\sigma_i$ of $L$ into disjoint cycles $O_{ij}$ 
\STATE fix $k$, such that $O_{ik}$ contains $i^{th}$ element of $Q$
\STATE $\mathcal{P'}=\text{PowerSet}\big\{O_{ij\mid j\neq k} \bigm| O_{ij},O_{ik}\subset Q\setminus \Sigma' ~\&~\lvert O_{ij}\rvert +\lvert O_{ik}\rvert \leq n/2   \big\}$
\STATE $\mathcal{P}=\Big\{W\in \mathcal{P'}\bigm| \lvert W\rvert +\lvert O_{ik}\rvert \leq n/2 \Big\}$
\WHILE{$\mathcal{P}\neq \emptyset$}
\STATE $ W\in \mathcal{P},\quad \mathcal{P}=\mathcal{P}\setminus \{W\}$
\STATE $W=O_{ik}\cup W$
\IF{ $W$ is closed under $\ast$} 
\STATE $\mathcal{L}=\mathcal{L} \cup \{W\}$
\ENDIF
\ENDWHILE
\STATE $S=\bigcup_{W\in \mathcal{L}} W\setminus \Sigma$
\WHILE{$S\neq \emptyset$}
\STATE decompose $\sigma_j,~\, j\in S,$
\STATE $\Sigma=\Sigma \cup \{j\}$
\STATE $\mathcal{P'}=\text{PowerSet}\big\{O_{jl\mid l\neq k}\bigm| O_{jl}, O_{jk}\subset S \text{ and }\lvert O_{jl}\rvert +\lvert O_{jk}\rvert \leq n/2  \big\}$
\STATE $\mathcal{P}=\Big\{W\in \mathcal{P'}\bigm| \lvert W\rvert +\lvert O_{jk}\rvert \leq n/2  \Big\}$
\WHILE{$\mathcal{P}\neq \emptyset$}
\STATE $ W\in \mathcal{P},\quad \mathcal{P}=\mathcal{P}\setminus \{W\}$
\STATE $W=O_{jk}\cup W$
\IF{ $W$ is closed under $\ast$} 
\STATE $\mathcal{L}=\mathcal{L} \cup \{W\}$
\ENDIF
\ENDWHILE
\STATE $S=S\setminus \{j\}$
\ENDWHILE
\STATE $\Sigma'=\Sigma' \cup \{i\}$
\STATE $i=\min\big\{ Q\setminus \Sigma \big\}$
\ENDWHILE 
\RETURN $\mathcal{L}$ 
\end{algorithmic}
\end{algorithm}
 Note that if the above algorithm returns empty list $\mathcal{L}$, then there is no any proper subquasigroup in $Q$. 
     \begin{prop}
Let  $D_1$
   be the set of all  diagonal entries  of  the Latin square of $Q$.   Define by induction $D_{k+1}$ as a set of all squares of elements from 
   $D_k$. Since $Q$ is finite there exists a positive integer $k$ such that $D=D_k=D_{k+1}$.  The quasigroup $Q$ has 
  no proper subquasigroups if and  only if $Q$ is generated by every element from $D$.   
\end{prop}
\begin{proof}
Suppose that $Q$ has no proper  subquasigroup and $a\in D$. Then $a$ generates $Q$. 

Converesely,
let $A$ be a subquasigroup in $Q$. 
If $a\in A$ then $a^2\in A\cap D_1$. By induction for any $k$ we can find an element $a_k
\in A\cap D_k$.    Hence by  assumption $A=Q$.
\end{proof}
\begin{example}
Consider a quasigroup $(Q=\{1,2,3,4,5\},\ast)$ as follows

\begin{align*} 
  \begin{array}{ c|ccccc}   
  \ast &  1 &  2 & 3 & 4 & 5 \\ 
  \hline
 1&  \boxed 2  &  3 & 1  & 4  & 5 \\  
2&  4 &   \boxed5&  3& 2& 1   \\  
 3 &   5 &1&  \boxed4& 2& 3 \\  
 4  &  1 & 2 &5 &  \boxed3  & 4 \\  
  5  & 3 &4& 2& 5& \boxed1   \\ 
  \end{array}
 \end{align*}
The elements of  diagonal set $D_1=\{ 2,5,4,3,1\}$  are  placed with in the box. The set of square of the diagonal elements  $D_2=D_1$, so $D=D_1$.  
Now, 
$$2\ast 2=5,~ (2\ast 2)\ast 2=5\ast 2=4,~ 2\ast (2\ast 2)= 2\ast 5=1,$$
$$(2\ast (2\ast 2))\ast 2= (2\ast 5)\ast 2=1\ast 2=3$$
Hence,  $2$ generates $Q$. 
Similary one can check that all other diagonal elements generates $Q$.  Note that here the order of parenthesis has to be taken into consideration.
\end{example}
\begin{example}
Consider the following quasigroup of order $8$.
\begin{align*}
 \begin{array}{ c|cccccccc} 
  \ast &  1 &  2 & 3 & 4 & 5 & 6 & 7&8\\ 
      \hline
  1&\boxed2& 5& 8& 3& 7& 6& 4& 1 \\ 
	2&3& \boxed1& 6& 2& 4& 8& 7& 5\\ 
	3&4& 6& \boxed1& 7& 3& 5& 2& 8\\ 
	4&8& 7& 2& \boxed6& 5& 3& 1& 4\\ 
	5&6& 4& 3& 8& \boxed1& 2& 5& 7\\ 
	6&5& 2& 7& 1& 8& \boxed4& 6& 3\\ 
	7&7& 8& 5& 4& 2& 1& \boxed3& 6\\ 
	8&1& 3& 4& 5& 6& 7& 8& \boxed2\\ 
 \end{array}
\end{align*}

 In this case, the set of diagonal elements $D_1=\{2,1,6,4,3\}$ is a proper subset of $Q$, $D_2=\{1,2,4,6\}$ and $D_3=D_2$, so $D=D_2$. We can check as above that any element of $D$ generates $Q$, hence $(Q,\ast)$ has no proper subquasigroup. The same has also been shown by Algorithm \ref{subqgalgo}.
\end{example}

Now, we present an algorithm for the construction of a subquasigroup 
generated by an element $a\in Q.$

\begin{algorithm}[H]
\caption{Generation of a subquasigroup by an element}
\label{sbqgElement}
\begin{algorithmic}[1]
\REQUIRE An element $a\in Q$
\ENSURE  Subquasigroup of $Q$ generated by $a$
\STATE   $A_{0}=\{a\}$
\STATE   $k=0$ 
\STATE   $A_{k+1}=A_{k}A_{k}\cup A_{k}=\{a\ast a, a\}$
\WHILE{$A_{k+1} \neq A_{k}$}
\STATE $k=k+1$
\STATE $A_{k+1}=A_{k}A_{k}\cup A_{k}$
\begin{eqnarray*}
\hspace{1.45 cm}=\left[\left(A_k\setminus A_{k-1}\right)\left(A_k\setminus A_{k-1}\right) \right]\cup 
\left[\left(A_k\setminus A_{k-1}\right)A_{k-1}\right] \cup \\
\hspace*{-6 cm}\left[A_{k-1}\left(A_k\setminus A_{k-1}\right)\right] \cup A_k
\end{eqnarray*}
\vspace{-0.2 cm}
\IF{$\vert A_{k+1}\vert > \vert Q\vert /2$}
\STATE $A_{k+1}=Q$
\STATE \textbf{break}
\ENDIF			
\ENDWHILE
\RETURN $(A_{k+1},\ast)$ 
\end{algorithmic}

\end{algorithm}
The above algorithm terminates since $Q$ is finite so there exists a positive integer $k$ such that $A_{k+1}=A_k$.  Note that all the possible order of parenthesis are considered in the above algorithm.

\begin{prop}
$ A_k$ is a subquasigroup in $Q$ generated by  an element $a\in Q$. Then, $a$ generates $Q$ if 
and only if $Q=A_k$. The quasigroup $Q$ has no proper subquasigroup if and only if $Q=A_k$ for every element $a\in Q$.
\end{prop}
\begin{proof}
If $A_{k+1}=A_k$ then by definition $A_k$ is closed under multiplication. Then, the result follows from  Proposition \ref{sq}.
\end{proof}
\begin{prop}
 Let $A_k$ be as above and $|A_k|>\frac{|Q|}2$. Then $A_k=Q$.
\end{prop}
\begin{remark}
It has to be noted that Algorithm \ref{subqgalgo} and Algorithm \ref{sbqgElement} both  can be used to  find whether  a given quasigroup does not have any proper subquasigroup. But using Algorithm \ref{sbqgElement}, one has to check the possibility of subquasigroups corresponding to each diagonal element of the quasigroup. So, in the case of no proper subquasigroup, which is of our interest, this algorithm will be  less efficient,   therefore confine its interest to the theoretical aspects.
\end{remark}

\section{ Finite Field Based Polynomially Complete Quasigroup of  Order $p^r$  Without  Proper Subquasigroups}\label{s3}
In this section, we propose a binary operation to construct quasigroups based on finite fields which gives polynomially complete quasigroups of order $p^r$, having no subquasigroups. We develop a methodology and an algorithm to construct such quasigroups. It provides cryptographically suitable quasigroups of order $2^r ~(r>1)$. Furthermore, we illustrate  some important properties of these quasigroups.

  Let  us consider a finite field $\mathbb{F}_{q}~(q=p^r)$ with $q>2$ elements, the set $Q=\{x\mid x\in \F_q\}$,  and $\alpha,\beta, c\in \F_q^*$.  
  Suppose that  $m,d$ are positive integers coprime  with $q-1$.
  Define a mutiplication in $Q$ as follows
  \begin{eqnarray} \label{qsq}
  x*y= \alpha x^m+\beta y^d+c
  \end{eqnarray}
  \begin{thm} 
  $(Q,\ast)$  is a quasigroup isotopic to the quasigroup $(Q,\cdot)$, where $\cdot$ is the multiplication defined in \eqref{qsq} 
  for $m=d=1$. The quasigroup $(Q,\cdot)$ is affine.
  \end{thm}
 \begin{proof}
 Since $m,d$ are coprime with $q-1$ the maps $\pi_1(x)= x^m$ and $\pi_2(y)=y^d$ are 
 permutations on ${Q}$. Hence,
 $$ x*y= \alpha \pi_1(x) + \beta \pi_2(y)+ c.$$
 That is, $$ x*y=\pi^{-1} (\pi_1(x)\cdot \pi_2(y))$$
 It follows that $(Q,*)$ is a quasigroup with necessary isotopy with $\pi$ as the identity map.
 \end{proof}
\begin{thm} \label{sqpc4}
 The quasigroup $(Q,*)$ from \eqref{qsq} is affine if and only if 
$md^{-1}\in \{1,p,\ldots,p^{r-1}\} \mod(q-1)$  where 
$p$ is the  characteristic of $\F_q,\; q=p^r$.  
\end{thm}
\begin{proof}
 Let $L_x$ and $R_y$ be the maps of left and right multiplications in the quasigroup $(Q,*)$ from 
 \eqref{qsq}. Then 
 $$
 L_x(y) =\beta y^d + \alpha x^m+c,\quad R_y(x) = \alpha x^m +\beta y^d+c.
 $$
 Hence 
 \begin{eqnarray*}
 &
 L_x^{-1}(y) =\sqrt[d]{\beta^{-1}y - \beta^{-1}\left(\alpha x^m+c\right)},
 \\ &
 R_y^{-1}(x) = \sqrt[m]{\alpha^{-1}x - \alpha^{-1}\left(\beta y^d+c\right)},
\end{eqnarray*}
Hence
\begin{eqnarray} \label{qq} 
L_xL_z^{-1}(y) =y+ \alpha (x^m-z^m), \quad 
R_yR_t^{-1}(x) = x + \beta (y^d-t^d).
\end{eqnarray}
There existrs a group isomorphism $\pi:G(Q,*)\to (\F_q,+)$ such that
$$
\pi\left(L_x,L_z^{-1}\right)  = \alpha(x^m-z^m),\quad \pi\left(R_yR_t^{-1}\right) = \beta(y^d-t^d).
$$
Suppose that $(Q,*)$ is affine. Then there exists a structure of an additive abelian group 
$(Q,  \oplus)  $  such that 
$$
x*y= \xi x \oplus \zeta y \oplus d \quad \xi,\zeta\in \Aut(Q,\oplus).
$$
Hence as in \eqref{qq}  
$$
L_xL_z^{-1}(y) = y\oplus \xi x\ominus \xi z,\quad 
R_yR_z^{-1}(x) = x \oplus \zeta y \ominus \zeta z, \quad \xi,\zeta\in \Aut(Q,\oplus).
$$
Again we get a group isomorphism $\omega:G(Q,*)\to (Q,\oplus)$  where
$$
\omega\left(L_xL_z^{-1}\right) = \xi x\ominus \xi z,\quad  \omega\left(R_yR_z^{-1}\right) =\zeta y \ominus \zeta z.
$$
Consequently,
$$
\omega\pi^{-1}\alpha(x^m-z^m) = \xi (x\ominus z),\quad \omega\pi^{-1}\beta( y^d-z^d) =\zeta(y\ominus z),
$$
or
$$
\xi^{-1}\omega\pi^{-1}\alpha(x^m-z^m) =  x\ominus z,\quad \zeta^{-1}\omega\pi^{-1}\beta( y^d-z^d) = y\ominus z,
$$
and therefore
$$
 \xi^{-1}\omega\pi^{-1} \alpha(x^m-z^m) =
\zeta^{-1}\omega \pi^{-1}\beta(x^d-z^d) = x\ominus z.
$$
It $z=0$ then $\xi^{-1}\omega\pi^{-1}\alpha(x^m)=x\ominus 0$ that is $\xi^{-1}\omega\pi^{-1}\alpha(x)=x^{1/m}\ominus 0$
and therefore 
$$
\xi^{-1}\omega\pi^{-1}\alpha(x^m-z^m)= (x^m-y^m)^{1/m} \ominus 0 = x\ominus z.
$$
Hence $(x^m-z^m)^{1/m} = x\ominus z\oplus 0$.

Similarly 
$$
\zeta^{-1}\omega \pi^{-1}\beta(x^d-z^d) = y\ominus z = (x^d-z^d) ^{1/d} \ominus 0,
$$
 Hence
$$
\alpha^{-1}\pi\omega^{-1} \xi\zeta^{-1} \omega \pi^{-1}\beta(x^d-z^d)= x^m-z^m.
$$
Recall that 
$
\xi\zeta^{-1} \in \Aut(Q,\oplus)$. Thus 
$$  
  \pi\omega^{-1}\xi\zeta^{-1}\omega\pi^{-1}   \text{ and }
 \theta=\alpha^{-1}\pi\omega^{-1} \xi\zeta^{-1} \omega \pi^{-1}\beta   \in \Aut(\F_q, +).
$$
Thus in the group $(\F_q,+)$ we have  $\theta(x^d-z^d)= x^m-z^m$ for some $\theta \in\Aut(\F_q,+).$
In particular $\theta(x^d)=x^m$ or $\theta(x)= x^{md^{-1}}$.  Since the map $\theta$ is additive, we can conclude that $md^{-1} $ is a power of $p$.

Conversely, let $m=dp^l,~0\leq l\leq r-1$. Then \eqref{qsq} has the form 
$$
x*y=\alpha x^{dp^l}+
\beta y^d+c.
$$
The map $\pi(x)= x^d$ defines an isomorphism of $(Q,*)$ and the quasigroup with multiplication 
$x\ast y= \alpha  \pi(x)^{p^l}+
\beta \pi(y)+c.
$ The last one is affine since the maps $x\mapsto \alpha  \pi(x)^{p^l}$
 and $y\mapsto \beta \pi(y)$ are automorphisms of $(\F_q,+)$.
\end{proof}

 \begin{thm} \label{sqpc}
Suppose that $Q$ is from \eqref{qsq} and $\beta$ is a generator of the cyclic
group $\F_q^*$.
Then $Q$ is simple.
\end{thm}
\begin{proof}
By \eqref{qq} the group $G(Q)$ contains all translations $x\mapsto x+w$ for all
$w\in \F_q$.  As it is shown in  the proof of Theorem \ref{sqpc4} 
$L_x(y)= \beta y^d+u,\; u = \alpha x^m+c\in \F_q$. Hence if we take $x$ such that $\alpha x^m+c=0$ then we see that the map $f(y) =\beta y^d$ belongs to 
$\Mult(Q)$. Consider the stabilizer subgroup  $H$ of zero element in $\Mult(Q)$.
It contains the map $f$ which acts transitively in $\F_q^*$ because $\beta$ is the
generator of $\F_q^*$.
Hence $\Mult(Q)$ is a doubly transitive permutation group and therefore $Q$ is simple,   \cite{Phillips}~[Proposition 1].
\end{proof}
\begin{cor}  \label{csq1}
Suppose that $md^{-1}\notin \{1,p,\ldots,p^{r-1}\} \mod(q-1)$ and $\beta$ is a generator of the cyclic
group $\F_q^*$. Then 
$Q$  with multiplication \eqref{qsq} is polynomially complete.  
\end{cor}
\begin{thm} \label{tn1}
 	Let $Q$ be the quasigroup defined in $\F_q, q>2,$  with multiplication 
 	$$
 	x*y=\alpha x^m + \beta y^d+c,\quad c\in \F_q, \quad \alpha,\beta\in \F_q^*,
 	$$
 	where $\beta$ is a generator of the cyclic group $\F_q^*$. Here $0<m,d<q-1$ are coprime with $q-1$.	
Then, any two distinct elements in $\F_q$ generate the quasigroup $Q$.
 \end{thm}
 \begin{proof}
 Consider in $\F_q$ new multiplication
 \begin{eqnarray} \label{mn}
 	x\odot y = \alpha x^{md^{-1}} + \beta y +c.
 \end{eqnarray}
 Denote by $L_x^{\odot}, L^*_x$ the operators of left multiplication by $x$  with respect to multplications $\odot, *$.   Since $x*y= x^d\odot y^d$ we have
 $L^*_x(y)= L^{\odot}_{x^d}(y^d)$. 
 
 Let's fix elements $x,y\in \F_q$.
 \begin{lem} \label{l1}
 	If $k\geqslant 1$ then $\left(L_x^{\odot}\right)^k (y) = \alpha\frac{1-\beta^k}{1-\beta} x^{md^{-1}} + \frac{1-\beta^k}{1-\beta} c + \beta^ky$.
 \end{lem}
 \begin{proof}
 The case $k=1$ follows from \eqref{mn}. Suppose that the formula is true for some $k$. Then by induction
 \begin{align*}
 	& \left(L_x^{\odot}\right)^{k+1} (y)  = L_x^{\odot}\left[ \alpha\frac{1-\beta^k}{1-\beta} x^{md^{-1}} + \frac{1-\beta^k}{1-\beta} c + \beta^ky\right]  \\
 	& = \alpha x^{md^{-1}} + \beta \left[ \alpha\frac{1-\beta^k}{1-\beta} x^{md^{-1}} + \frac{1-\beta^k}{1-\beta} c + \beta^ky\right]  + c  \\
 	& = \alpha\left[ 1+ \beta \frac{1-\beta^k}{1-\beta} \right]x^{md^{-1}}  + \beta^{k+1} y + \left[ 1+ \beta \frac{1-\beta^k}{1-\beta} \right] c.
 \end{align*}
 It suffices to notice that  $1+ \beta \frac{1-\beta^k}{1-\beta}= \frac{1-\beta^{k+1}}{1-\beta}$.
 \end{proof}
 \begin{lem} \label{l3}
 	Let $W$ be a subquasigroup of $(Q,*)$ containing elements ${ x,y}$. If $Q\ne W$, then 
 	\begin{eqnarray} \label{ng}
 		\alpha x^{m} + (\beta-1)y^d +c =0.
 	\end{eqnarray}
 \end{lem}
 \begin{proof}
 As we have already noticed $L^*_x(y)= L^{\odot}_{x^d}(y^d)$.  Suppose that  $ 	\left(L_x^{*}\right)^k (y)  =y$ for some { minimal} $1\leqslant k\leqslant q-2$.
 Then $ 	\left(L_{x^d}^{\odot}\right)^{k+1}( y^d)  = L_{x^d}^{\odot}  (y^d)$. By Lemma \ref{l1},
 $$
 \alpha\left[ \frac {1-\beta^{k+1}}{1-\beta}\right] x^m + \beta^{k+1} y^d + \frac{1-\beta^{k+1}}{1-\beta} c =    \alpha x^m+\beta y^d+c.
 $$
 Hence 
 $$
 \alpha\left[ \frac {1-\beta^{k+1}}{1-\beta}-1\right] x^m + \left[\beta^{k+1}-\beta\right]y^d + \left[\frac{1-\beta^{k+1}}{1-\beta}-1\right] c = 0
 $$
 Note that $ \frac {1-\beta^{k+1}}{1-\beta}-1= \frac{\beta(1-\beta^k)}{1-\beta}.$
 Multipliying by $\frac{1-\beta}{\beta(1-\beta^k)}$, we obtain  \eqref{ng}. Note that the order of $\beta$ is equal to $q-1.$ So if  $k\leqslant q-2$ then $1\ne \beta^k$.
 
 Suppose  that $k=q-1$, Then  $W$ contains  distinct elements
  $${ y,(L^*_{x})(y), (L^*_{x})^2(y),\ldots, (L^*_{x})^{q-2}(y)}$$ 
{ Thus, $W=Q$ since $q-1>\frac q2$ for $q>2.$} 
 \end{proof}
 
 Suppose that $W$ is a proper subquasigroup of $(Q,*)$ containing elements $x,y$.
Then \eqref{ng} holds for $x,y\in W$.
Let $W'$ be a subquasigroup generated by $x$. Then $W'\subseteq W$ and therefore $W'\ne Q$.  In particular, on  taking  $y=x$ in   \eqref{ng} we obtain
\begin{eqnarray}\label{xng}
\alpha x^m+ (\beta-1)x^d+c=0.
\end{eqnarray}
That is for an element $x\in W$ the equations \ref{xng} holds.
By combining \ref{xng} for $x\in W$ with \eqref{ng} for the same $x$ and another element $y\in W$, we obtain 
$(\beta-1)x^d=( \beta-1)y^d$.  
Since $1-\beta\ne 0$, we can conclude that $y=x$ because $d$ is coprime with $q-1$.\\ 
Thus, if $W$ is a proper subquasigroup of $(Q,*)$, it is an idempotent. That is, any two distinct elements generate the $(Q,*)$. 
\end{proof}
\begin{cor} \label{corConst}
Let $p$ be a prime and $q=p^r$.
Suppose that $m\notin \{1,p,\ldots,p^{r-1}\} \mod(q-1)$  and $\beta$ is a generator of the cyclic
group $\F_q^*$.   Suppose that $1<m<q-1$ is coprime with $q-1$ and $d=1$.
  Then there exists an element 
$c\in \F_q^*$ such that  $ Q=\{x \mid x\in \F_q\}$ with multiplication
\begin{eqnarray} \label{eq11}
 x*y = (1-\beta) x^m +\beta y +c
 \end{eqnarray}
   has no subsquasigroups and it is polynomially complete. An element $x\in Q$ form a subquasigroup if and only if 
\begin{eqnarray} \label{eqr1}
 x^m -x + \frac c{1-\beta} =0.
 \end{eqnarray}
\end{cor}
\begin{proof}
 The corollary follows immediately from the previous  theorem \ref{tn1} except the existence of an element $c\in \F_q*$ such that $(Q,*)$ does not have any proper subquasigroup.
\\ Let's consider a map
$$f:\FF_q\rightarrow \FF_q,\quad x\mapsto f(x)=x^m-x $$
Since $f(0)=f(1)=0$, the map $f$ is not bijective while the map
$$g:\FF_q\rightarrow \FF_q,\quad c\mapsto g(c)=c/(\beta -1)$$
is one-one. So, there exists a constant $c\in \FF_q^{\star}$ such that $c/(\beta -1) \in \FF_q^{\star}\setminus\text{Range}(f)$, hence Equation \eqref{eqr1} has no solution  in $\FF_q$, consequently, $(Q,\ast)$ has no proper subquasigroup.

\end{proof}

 Furthermore, the number of roots of the polynomial in \eqref{eqr1} can be found by using K\"onig-Rados 
  theorem, given in \cite{Lidl1996}~ [Chapter 6.1],  which states that the equation \eqref{eqr1} will have  $q-1-r$ number of solutions in $\mathbb{F}_{q}$, where $r$ is the rank of left circulant matrix corresponding to polynomial in \eqref{eqr1} of order $(q-1)\times (q-1)$
\begin{center}
  $$
  \begin{pmatrix}
  \gamma  & -1 & 0 & \hdots & 0 & 1 & 0 & \hdots & 0 & 0  \\
  -1 & 0 &   \hdots & 0 & 1 & 0 & \hdots   & 0 & 0 & \gamma\\
      \cdot & \hdots &   \cdot & \cdot & \cdot & \hdots & \cdot   & \cdot & \cdot & \cdot \\
    \hdots &   \cdot & \cdot & \cdot & \hdots & \cdot   & \cdot & \cdot & \cdot &\cdot \\
    0 &  \gamma & -1 &0 & \hdots & 0 & 1 & 0 & \hdots & 0  \\
  \end{pmatrix}, \quad \text{where}~\gamma = \frac{c}{1-\beta}.
  $$
\end{center} 

\begin{remark}

There are only $q-1-r$ idempotent subquasigroups in $(Q, \ast)$ for any $c\in \FF_q^{\star}$. Precisely, they are $(\{x\},\ast)$, where $x$ is a solution of Equation \eqref{eqr1}.   
\end{remark}

Based on Corollary \ref{corConst}, we propose the following  algorithm for the construction of polynomially complete quasigroups of order $p^r$ having no proper subquasigroup.
\begin{algorithm}[H]
\caption{Construction of Cryptographically Suitable Quasigroups of order $q=p^r$ }
\label{qgalgo}
\begin{algorithmic}[1]
\REQUIRE A prime $p$  and a positive integer $r$
\ENSURE A polynomially complete quasigroup of order $p^r$ having no proper subquasigroups
\STATE select an integer $m \notin  \{1,p,p^2,\ldots,p^{r-1}\}, ~1< m < (q-1)$, and\\ 
\hspace{2.7 cm}  $\text{gcd}(m,q-1)=1 $ 
\STATE fix a generator $\beta$ of  $\FF _{q}^{\star}$
\WHILE{True}
\STATE choose an element $c\in \FF _{q}^{\star}$
\STATE compute rank $r$ of the left circulant matrix of $x^m-x+\frac{c}{1-\beta}=0$ 
\IF{$q-1= r$}
\STATE \textbf{break}
\ENDIF  
\ENDWHILE
\STATE $Q=\{x \mid x\in \FF _{q} \}$
\STATE define binary operation $\ast$ on $Q$:
 \[\quad x \ast y = (1-\beta)x^m+\beta y + c, ~ \forall x,y \in Q \]
  \vspace{-4.7mm}
\RETURN $(Q,\ast)$ 
\end{algorithmic}

\end{algorithm}
\begin{remark}
One can avoid computing the rank of left circulant matrix at Step 5, in the above algorithm by determining the \emph{ range} of the map $f:\FF_q\rightarrow \FF_q,\quad x\mapsto x^m-x$. Since the map $f$ is not bijective, the nonempty set  $\FF_q\setminus \text{Range }(f)$  provides all those $c$ for which there will not be any subquasigroup. 
\end{remark}

Now, we give some concrete examples of construction of quasigroups of order $8$  over the field $\mathbb{F}_{2^3}$ to illustrate the above described method. For these constructions, possible choices of the integer $m$, mentioned in the algorithm, are $3,5,6$. In the following examples, we denote a \emph{primitive element} of $\mathbb{F}_{q}$  by $a$ and for convenience, tag the  elements $0,1,a,a+1,a^2,a^2+1,a^2+a,a^2+a+1$ of $\FF_q$ by $0,1,2,3,4,5,6,7$, respectively.  

\begin{example} \label{exConst1}Let $c =  1 \in \mathbb{F}_{2^3}$,  $m=3$, and a generator $\beta=a$ of  $\mathbb{F}_{2^3}^{\star} $. For these choices of $c,m $ and $\beta$, the set $Q=\{x \mid x\in \mathbb{F}_{2^3}\}$, with respect to binary operation $\ast$ as given in \eqref{eq11}, yields the following  polynomially complete quasigroup  
\begin{align*}
 \begin{array}{ c|cccccccc} 
  \ast&0 &  1 &  2 & 3 & 4 & 5 & 6 & 7\\ 
      \hline
      0 & 1 & 3 & 5 & 7 & 2 & 0 & 6 & 4  \\  
      1 & 2 & 0 & 6 & 4 & 1 & 3 & 5 & 7 \\  
      2 & 4 & 6 & 0 & 2 & 7 & 5 & 3 & 1 \\  
      3 & 6 & 4 & 2 & 0 & 5 & 7 & 1 & 3 \\  
      4 & 5 & 7 & 1 & 3 & 6 & 4 & 2 & 0\\  
      5 & 0 & 2 & 4 & 6 & 3 & 1 & 7 & 5\\  
      6 & 3 & 1 & 7 & 5 & 0 & 2 & 4 & 6\\ 
      7 & 7 & 5 & 3 & 1 & 4 & 6 & 0 & 2 \\ 
 \end{array}
\end{align*}

Now, the left circulant matrix corresponding to \eqref{eqr1} is  

$$
  \begin{pmatrix}
  a^2+a  & 1 & 0 & 1 & 0 & 0 & 0  \\
  1  & 0 & 1 & 0 & 0 & 0 & a^2+a  \\
  0  & 1 & 0 & 0 & 0 & a^2+a & 1  \\
  1  & 0 & 0 & 0 & a^2+a & 1 & 0  \\
  0  & 0 & 0 & a^2+a & 1 & 0 & 1  \\
  0  & 0 & a^2+a & 1 & 0 & 1 & 0  \\
  0  & a^2+a & 1 & 0 & 1 & 0 & 0  \\
  \end{pmatrix}  
  $$
whose rank is $7$. So, Equation (\ref{eqr1}) has no solution in  $\mathbb{F}_{2^3}$, hence the constructed quasigroup, $(Q, \ast)$, has no proper subquasigroup.   
\end{example}
\begin{example}Now, let $c =  a^2 \in \mathbb{F}_{2^3}$,  $m=5$, and a generator $\beta=a$ of  $\mathbb{F}_{2^3}^{\star} $. For these choices of $c,m $ and $\beta$, the set $Q=\{x \mid x\in \mathbb{F}_{2^3}\}$, with respect to binary operation $\ast$ as given in \eqref{eq11}, yields the following  polynomially complete quasigroup  
\begin{align*}
 \begin{array}{ c|cccccccc} 
  \ast&0 &  1 &  2 & 3 & 4 & 5 & 6 & 7\\ 
      \hline
      0 & 4&6&0&2&7&5&3&1 \\
      1 & 7&5&3&1&4&6&0&2  \\    
      2 & 6&4&2&0&5&7&1&3 \\
      3 & 2&0&6&4&1&3&5&7\\
      4 & 1&3&5&7&2&0&6&4 \\    
      5 & 3&1&7&5&0&2&4&6\\  
      6 & 0&2&4&6&3&1&7&5\\
      7 & 5&7&1&3&6&4&2&0 \\  
 \end{array}
\end{align*}
Note that  the above quasigroup is only permutation of some rows in the quasigroup constructed in  Example \ref{exConst1}.

The left circulant matrix corresponding to \eqref{eqr1} is  

$$
  \begin{pmatrix}
  a^2+1  & 1 & 0 & 1 & 0 & 0 & 0  \\
  1  & 0 & 1 & 0 & 0 & 0 & a^2+1  \\
  0  & 1 & 0 & 0 & 0 & a^2+1 & 1  \\
  1  & 0 & 0 & 0 & a^2+1 & 1 & 0  \\
  0  & 0 & 0 & a^2+1 & 1 & 0 & 1  \\
  0  & 0 & a^2+1 & 1 & 0 & 1 & 0  \\
  0  & a^2+1 & 1 & 0 & 1 & 0 & 0  \\
  \end{pmatrix}  
  $$
whose rank is $6$. So, Equation (\ref{eqr1}) has one solution in  $\mathbb{F}_{2^3}$. The element $2$, that is, $a$ satisfy the Equation (\ref{eqr1}). Hence, the constructed quasigroup, $(Q, \ast)$, has one idempotent subquasigroup $(\{2\}, \ast)$. 
\end{example}
\begin{remark}
All the quasigroups constructed by varying the constant $c\in \FF_q^{\star}$ and integer $m$ while keeping  the generator $\beta$ of $\FF_q^{\star}$ same in Equation (\ref{eq11}), belong to the same isotopic class. 
\end{remark}

  In connection with a consideration of subsquasigroups  it is necessary to mention the following.
\begin{thm}[\cite{kepka1978}] 
  A countable quasigroup with at least three members is necessarily isotopic to a quasigroup which has no proper subquasigroups.
  \end{thm}

Consider associative triples $x,y,z$ in the quasigroup ${ Q=\{x\mid x\in \F_q\}}$ with multiplication \eqref{qsq} from Corollary \ref{corConst}. It means that 
$(x*y)*z=x*(
y*z)$. 
Then we  have
\begin{multline*}
(1-\beta)\left((1-\beta)x^m+\beta y+c\right)^{m}+ \beta z+c = \\
(1-\beta)x^m+\beta \left((1-\beta)y^m+\beta z+c\right)+c,
\end{multline*}

$$
\beta(1-\beta) z = (1-\beta) x^m +\beta(1-\beta) y^m -(1-\beta)[(1-\beta)x^m+\beta y+c]^m+ \beta c
$$
Since $\beta(1-\beta)\ne 0$ we obtain,
$$
z = \frac{x^m}{\beta}+y^m-\frac{[(1-\beta)x^m+\beta y+c]^m}{\beta} +\frac{c}{1-\beta}.
$$
 So the pair $x,y$ uniquely determines $z$. Hence we prove  the following result.
 \begin{thm} The number of associative triples in the quasigroup $Q$ constructed from 
 Corollary \ref{corConst}  is  equal to $q^2=|Q|^2.$
 \end{thm}
 
\section{Conclusion}

Identifying the polynomially complete quasigroups with no proper subquasigroup is essential for the design and analysis of the quasigroup based crypto primitives and schemes. In this context, we have developed an effective algorithm that concludes if a given quasigroup has no proper subquasigroup, otherwise lists its all proper subquasigroups. We have constructed a class of quasigroups of order $p^r$ by defining a binary operation based on arithmetic of finite fields. Further, we have given the criteria under which the quasigroups of this class are polynomially complete and have no proper subquasigroup. This work provides an effective technique to construct cryptographically suitable quasigroups. We have also demonstrated the effectiveness of our methods by their implementations in Singular on some examples.

In the future work, we aim to construct cryptographic S-boxes by using quasigroups of class presented herein, and then analyze their cryptographic properties with respect to the S-boxes based on arbitrarily chosen quasigroups.
Moreover, we will also compare the cryptographic properties of these S-boxes with the benchmark S-boxes of AES\cite{aes}.

\section{Acknowledgement} 
We are thankful to  Ms. Anu Khosla, Director SAG, DRDO and Dr. Sudhir Kamath, DG, MED$\&$CoS, DRDO  for their supports and encouragements to carry out this collaborative research work. 
 Authors are also thankful to all the team members of Indo-Russian joint project QGSEC for their technical supports and scientific discussions.

\bibliographystyle{spbasic}
\bibliography{ConstructionOfCryptographicallySuitableQuasigroups}

\begin{thebibliography}{26}
\providecommand{\natexlab}[1]{#1}
\providecommand{\url}[1]{{#1}}
\providecommand{\urlprefix}{URL }
\expandafter\ifx\csname urlstyle\endcsname\relax
  \providecommand{\doi}[1]{DOI~\discretionary{}{}{}#1}\else
  \providecommand{\doi}{DOI~\discretionary{}{}{}\begingroup
  \urlstyle{rm}\Url}\fi
\providecommand{\eprint}[2][]{\url{#2}}

\bibitem[{{Artamonov}(2012)}]{Artamonov2012}
{Artamonov} VA (2012) {Polynomially complete algebras.} {Series Natural Tech
  Med Sci} 6(2):23--29

\bibitem[{{Artamonov}(2020)}]{Artamonov2020}
{Artamonov} VA (2020) { Automorphisms of finite quasigroups with no proper no
  subquasigroups.} {Vestnik StPetersbourg university, Mathematics, Mechanics,
  Astronomy}

\bibitem[{{Artamonov} et~al.(2013){Artamonov}, {Chakrabarti}, {Gangopadhyay},
  and {Pal}}]{Artamonov2013}
{Artamonov} VA, {Chakrabarti} S, {Gangopadhyay} S, {Pal} SK (2013) {On Latin
  squares of polynomially complete quasigroups and quasigroups generated by
  shifts.} {Quasigroups Relat Syst} 21(2):117--130

\bibitem[{{Artamonov} et~al.(2016){Artamonov}, {Chakrabarti}, and
  {Pal}}]{Artamonov2016}
{Artamonov} VA, {Chakrabarti} S, {Pal} SK (2016) {Characterization of
  polynomially complete quasigroups based on Latin squares for cryptographic
  transformations.} {Discrete Appl Math} 200:5--17

\bibitem[{{Artamonov} et~al.(2017){Artamonov}, {Chakrabarti}, and
  {Pal}}]{Artamonov2017}
{Artamonov} VA, {Chakrabarti} S, {Pal} SK (2017) {Characterizations of highly
  non-associative quasigroups and associative triples.} {Quasigroups Relat
  Syst} 25(1):1--19

\bibitem[{{Belyavskaya}(1989)}]{Belyavskaya1989}
{Belyavskaya} GB (1989) {T-quasi-groups and the center of a quasi-group.} {Mat
  Issled} 111:24--43

\bibitem[{{Belyavskaya}(1994)}]{Belyavskaya1994}
{Belyavskaya} GB (1994) {Abelian quasigroups are \(T\)-quasigroups.}
  {Quasigroups Relat Syst} 1(1):1--7

\bibitem[{{Belyavskaya} and {Tabarov}(1992)}]{Belyavskaya1992}
{Belyavskaya} GB, {Tabarov} AK (1992) {Characteristic of linear and alinear
  quasigroups.} {Diskretn Mat} 4(2):142--147

\bibitem[{Daemen and Rijmen(2002)}]{aes}
Daemen J, Rijmen V (2002) The Design of Rijndael: AES - The Advanced Encryption
  Standard (Information Security and Cryptography), 1st edn. Springer

\bibitem[{Decker et~al.(2019)Decker, Greuel, Pfister, and
  Sch\"onemann}]{singular}
Decker W, Greuel GM, Pfister G, Sch\"onemann H (2019) {\sc Singular} {4-1-2}
  --- {A} computer algebra system for polynomial computations.
  \url{http://www.singular.uni-kl.de}

\bibitem[{{D\'enes} and {Keedwell}(1991)}]{Denes1991}
{D\'enes} J, {Keedwell} AD (eds)  (1991) {Latin squares. New developments in
  the theory and applications.}, vol~46. Amsterdam etc.: North-Holland

\bibitem[{Dimitrova and Markovski(2004)}]{Dimitrova04}
Dimitrova V, Markovski J (2004) On quasigroup pseudo random sequence generator.
  In: Proc. of the 1-st Balkan Conference in Informatics, Thessaloniki, pp
  393--401

\bibitem[{{Galatentko} et~al.(2018){Galatentko}, {Pankrat'ev}, and
  {Rodin}}]{Galatentko2018}
{Galatentko} AV, {Pankrat'ev} AE, {Rodin} SB (2018) {Polynomially complete
  quasigroups of prime order.} {Algebra Logic} 57(5):327--335

\bibitem[{{Gro\v{s}ek} and {Hor\'ak}(2012)}]{Otokar2012}
{Gro\v{s}ek} O, {Hor\'ak} P (2012) {On quasigroups with few associative
  triples.} {Des Codes Cryptography} 64(1-2):221--227

\bibitem[{{Hagemann} and {Herrmann}(1982)}]{Hagemann1982}
{Hagemann} J, {Herrmann} C (1982) {Arithmetical locally equational classes and
  representation of partial functions.} {Universal algebra, Proc. Colloq.,
  Esztergom/Hung. 1977, Colloq. Math. Soc. Janos Bolyai 29, 345-360 (1982).}

\bibitem[{{Horv\'ath} et~al.(2008){Horv\'ath}, {Nehaniv}, and
  {Szab\'o}}]{Horv2008}
{Horv\'ath} G, {Nehaniv} CL, {Szab\'o} C (2008) {An assertion concerning
  functionally complete algebras and NP-completeness.} {Theor Comput Sci}
  407(1-3):591--595

\bibitem[{{Keedwell} and {D\'enes}(2015)}]{Keedwell2015}
{Keedwell} AD, {D\'enes} J (2015) {Latin squares and their applications. 2nd
  ed.}, 2nd edn. Amsterdam: Elsevier

\bibitem[{{Kepka}(1978)}]{kepka1978}
{Kepka} T (1978) {A note on simple quasigroups.} {Acta Univ Carol, Math Phys}
  19(2):59--60

\bibitem[{{Kepka} and {Nemec}(1971)}]{kepka1971}
{Kepka} T, {Nemec} P (1971) {Affine quasigroups.} {Acta Univ Carol, Math Phys}
  12(1):39--49

\bibitem[{{Lidl} and {Niederreiter}(1996)}]{Lidl1996}
{Lidl} R, {Niederreiter} H (1996) {Finite fields. 2nd ed.}, vol~20, 2nd edn.
  Cambridge: Cambridge Univ. Press

\bibitem[{{Markovski} et~al.(1997){Markovski}, {Gligoroski}, and
  {Andova}}]{smile1997}
{Markovski} S, {Gligoroski} D, {Andova} S (1997) {Using quasigroups for one-one
  secure encoding.} In: {Proceedings of the VIII international conference on
  logic and computer science: theoretical foundations of computer science, Lira
  '97, Novi Sad, Yugoslavia, September 1--4, 1997}, Novi Sad: Univ. of Novi
  Sad, Faculty of Science, Institute of Mathematics, pp 157--162

\bibitem[{{Markovski} et~al.(1999){Markovski}, {Gligoroski}, and
  {Bakeva}}]{Markovski1999}
{Markovski} S, {Gligoroski} D, {Bakeva} V (1999) {Quasigroup string
  processing-part 1.} {Contributions, Sec Math Tech Sci MANU} XX:13--28

\bibitem[{Phillips and Smith(1999)}]{Phillips}
Phillips J, Smith J (1999) Quasiprimitivity and quasigroups. Bulletin of the
  Australian Mathematical Society 59(3):473–475,
  \doi{10.1017/S0004972700033165}

\bibitem[{{Shcherbacov}(2017)}]{viktor2017}
{Shcherbacov} VA (2017) {Elements of quasigroup theory and applications.} Boca
  Raton, FL: CRC Press

\bibitem[{{Smith}(2007)}]{smith2007}
{Smith} JDH (2007) {An introduction to quasigroups and their representations.}
  Boca Raton, FL: Chapman \& Hall/CRC

\bibitem[{Wall and Drury(1957)}]{Wall1957}
Wall W, Drury W (1957) Subquasigroups of finite quasigroup. {Pacific J of
  Mathematics} 7(4):1711--1714

\end{thebibliography}

\end{document}